\documentclass[12pt]{amsart}
\usepackage{latexsym,fancyhdr,amssymb,color,amsmath,amsthm,graphicx,listings,comment}
\usepackage[section]{placeins}  \newtheorem{thm}{Theorem} 
 \let\paragraph\subsection
\title{The Colorful Ring of Partitions}
\author{Oliver Knill}
\date{9/13/2024}
\address{Department of Mathematics \\ Harvard University \\ Cambridge, MA, 02138 }
\begin{document}
\maketitle


\begin{abstract}
We visualize the identity 
$p(n) = \sum_{k=1}^n \sigma(k) p(n-k)/n$ for the partition function 
$p(n)$ involving the divisor function $\sigma$, add comments on the history
of visualizations of numbers, illustrate how different mathematical
fields play together when proving $\lim_{n \to \infty} p(n)^{1/n}=1$ and 
introduce the modular ring of partitions. 
\end{abstract} 

\section{Introduction}

\paragraph{}
As far as we know, the earliest mathematicians worked with the 
help of {\bf manipulatives} in the form of sticks, pebbles and ropes. 
Finger or toe cardinalities might have initiated the {\bf decimal} or 
{\bf base 20} systems, divisibility considerations the hexadecimal number system.
Marks were carved into bones, clay, stone or 
bark \cite{Katz2007,guinness,boyer}. 
Numbers were then written on Clay tablets or Papyrus or
encoded in knots \cite{Locke1912}. While historical discoveries are correlated to
pedagogical time lines \cite{Eves}, history is always an effective 
ingredient for teaching \cite{Katz2000}. Manipulatives are not only relevant 
historically or for teaching, they help to build intuition, even in a time of 
smartphones, tablets, virtual reality and artificial intelligence. 
The use of physical manipulative has flared up in the 
3D printing age \cite{CFZ,Segerman2016}. Dubbed as the {\bf forth 
industrial evolution} \cite{Rifkin}, it was recently over shadowed by 
the AI revolution \cite{AttentionIsAllYouNeed}.
Manipulatives might reappear more in virtual reality set-ups, once interfaces 
become fully VARK: {\bf visual}, {\bf auditory}, {\bf read related} 
or {\bf kinesthetic} \cite{VARK}.

\begin{figure}[!htpb]
\scalebox{0.03}{\includegraphics{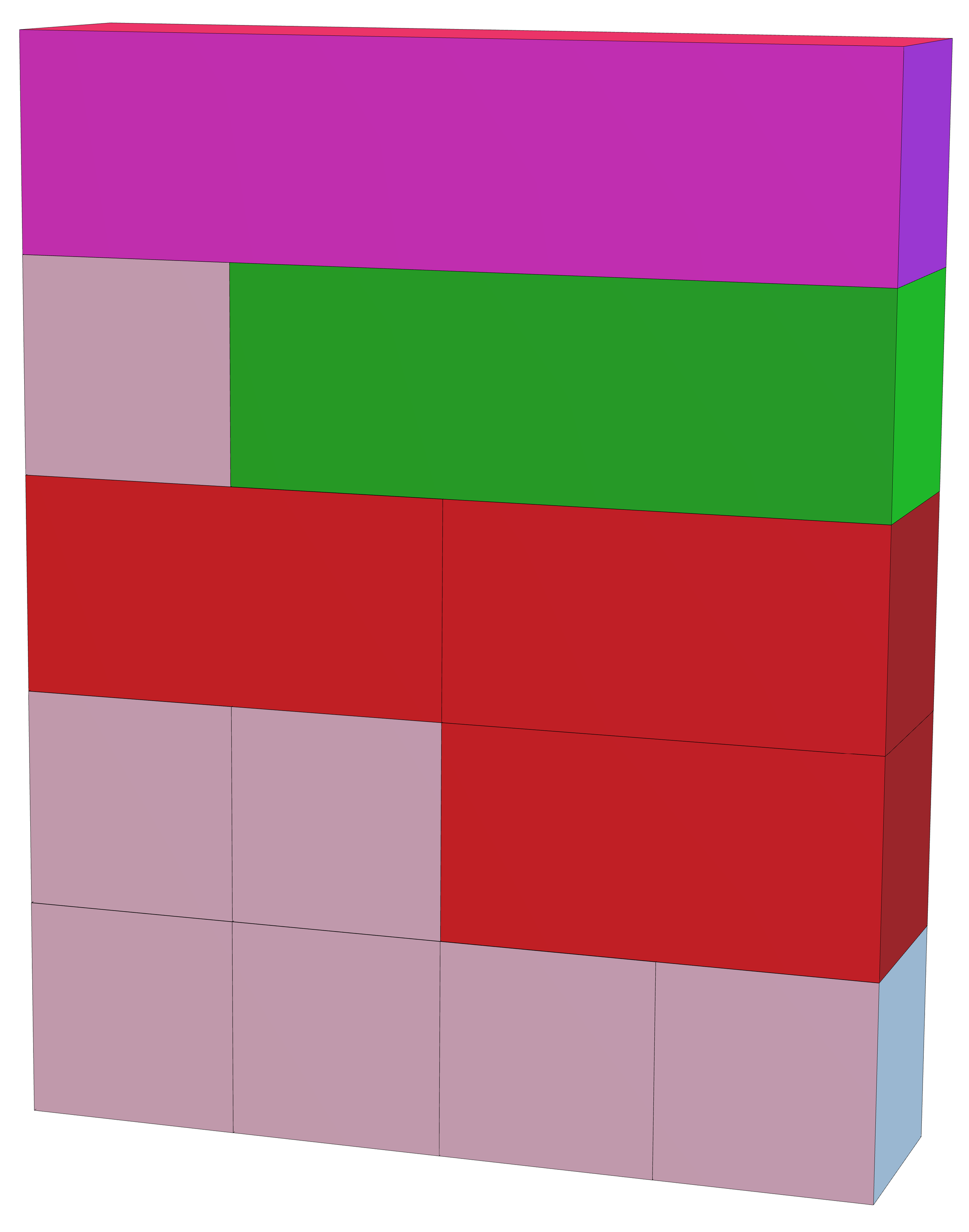}}
\hspace{9mm}
\scalebox{0.45}{\includegraphics{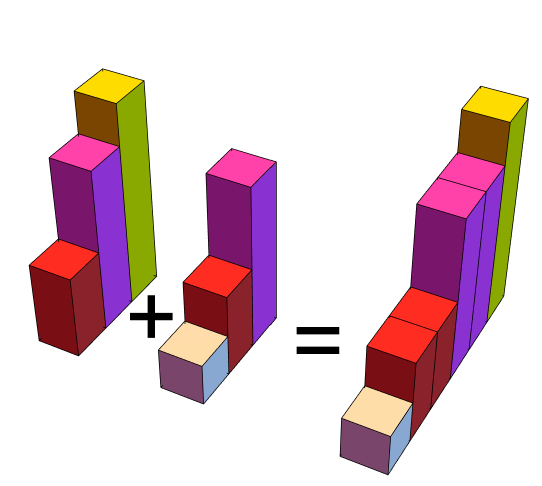}}
\scalebox{0.45}{\includegraphics{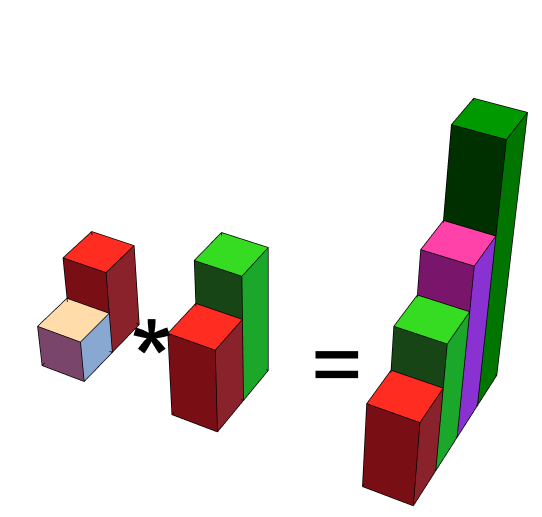}}
\label{Partitions}
\caption{
Partitions like the $p(4)=5$ partitions of $5$ generalize numbers. 
There is addition and multiplication in the ring $(\mathcal{P},+,*,0,1)$. Note that $1*1=1$
but $1+1=(1,1) \neq (2)$. While $(2,2)=(2)+(2)$, the partition $(4)$ is not the sum of
smaller partitions; it is an {\bf additive partition prime}. Every $n=(n_1,\dots,n_k)$ with either
$|n|=\sum_{j=1}^k n_j$ or $k$ a rational prime is a {\bf multiplicative partition prime} 
but there are more: $(1,2,3,4)$ is a multiplicative partition prime, even so $|n|=1+2+3+4$, 
and $k=4$ are not prime. 
}
\end{figure}

\section{Partitions as numbers}

\paragraph{}
To illustrate how manipulatives can lead to {\bf innovation}, we look at
partitions as ``generalized numbers" enlarging ``natural numbers"
that initiate the build-up
$\mathbb{N} \subset \mathbb{Z}  \subset \mathbb{Q} \subset \mathbb{R} \subset 
\mathbb{C} \subset \mathbb{H} \subset \mathbb{O}$. The generalization is geometric:
we can ``add "finite simple graphs using the join operation \cite{Zykov}. 
After completing the monoid $(\mathcal{G},\oplus)$  to a group, it leads 
with the large multiplication \cite{Sabidussi} to the {\bf Sabidussi ring} 
$(\mathbb{G},\oplus,\otimes,0,1)$ in which the empty graph is the zero element and 
the $1$-point graph $1$ is the unit. 
This ring is isomorphic to the {\bf Shannon ring} $(\mathcal{Z},+,*,0,1)$, 
in which $+$ is the {\bf disjoint union} and $*$ is the 
{\bf Shannon multiplication} \cite{Shannon1956}. See 
\cite{HammackImrichKlavzar,Shemmer,VisualizationGraphProducts,ArithmeticGraphs,
StrongRing, numbersandgraphs}. 

\paragraph{}
The upshot is that there is a build-up $\mathcal{N} \subset \mathcal{Z}  
\subset \mathcal{Q} \subset \mathcal{R} \subset \mathcal{C} \subset \mathcal{H} \subset \mathcal{O}$
also for partitions $\mathcal{P} \sim \mathcal{Z}$. 
We do not have division algebras because the list of real normed division algebras is limited to to the 
algebras $\mathbb{R},\mathbb{C},\mathbb{H},\mathbb{O}$ by {\bf Hurwitz theorem} 
\cite{Hurwitz1922} and $\mathcal{R}$ contains $\mathbb{R}$ so that
also $\mathcal{C},\mathcal{H},\mathcal{O}$ would be division algebras which is impossible. 
However, the {\bf partition complex numbers} $\mathcal{C}$ carry a {\bf Banach algebra structure}.
This can be done for any set of networks as generators. The simplest is to start with 
one single network which then completes to the completes to a Banach algebra isomorphic 
to the {\bf Wiener algebra} $A(\mathbb{T})$. See \cite{RemarksArithmeticGraphs}. 

\paragraph{}
Partitions are relevant in the context of the {\bf discrete Sard theorem} \cite{ManifoldsPartitions}:
if we take a discrete $d$-manifold $G$ and a $(k+1)$-partite graph $P=K_{n_0,\dots, n_k}$ and a
function $f: V(G) \to V(P)$ such that all facets $F$ in $P$ are reached, then the {\bf level 
surface} $\{ x \in G, f(x)$ contains at least an element in $F \}$, is either empty or then a 
$(d-k)$-manifold. Always. This is a remarkable result given that in the continuum, level sets 
can be singular like $x^2+y^2-z^2=1$. 

\paragraph{}
{\bf The ring of partitions} $\mathcal{P}$ is a subring of the Sabidussi ring 
$\mathcal{S}$. Its elements are multi-partite graphs, elements generated by zero
dimensional subgraphs.  
The integers $\mathcal{Z}$ are the subring of the Sabidussi ring generated by complete graphs. 
In the partition ring, the concept of a number like "5" is richer. 
It contains $7$ flavors of $5$, the $7$ partitions of $5$. 
The addition is the concatenation. This comes up naturally in the
Cuisenaire picture \cite{CuisenaireGattegno}. 
The multiplication can be visualized geometrically also.

\begin{thm}
The ring of partitions $\mathcal{P}$ is an extension of the ring $\mathbb{Z}$ of integers and
contained in the Sabidussi ring of graphs $\mathcal{G}$ which is isomorphic to the 
Shannon ring of graphs $\mathcal{Z}$. A partition $n$ that adds up to a rational prime
$p=|n|$ or contains a rational prime $k(n)=p$ number of elements is a multiplicative 
partition prime.  
\end{thm} 
\begin{proof}
$\mathcal{Z}$ is generated by complete graphs. $\mathcal{P}$ is generated by 
$0$-dimensional graphs. Note that $K_1 \oplus K_1 \cdots \oplus K_1=K_n$ is a 
complete graph so that $\mathcal{Z} \subset \mathcal{P}$. $\mathcal{G}$ is generated by all graphs.
The duality in the form of {\bf graph complement} represents $\mathcal{Z}$ in the Shannon ring
as generated by zero dimensional graphs and $\mathcal{P}$ generated by complete 
graphs and $\mathcal{G}$ generated by all graphs. 
To see the prime statement, 
note that both $n=(n_1, \dots, n_k) \to |n|=\sum_{j=1}^{k(p)} n_j$ and $n \to k(n)$ are
{\bf ring homomorphisms} from $\mathcal{P}$ to $\mathbb{Z}$. 
The homomorphism statements follow from 
$|n+m| = |n|+|m|, k(n+m)=k(n)+k(m), |n*m|=|n| |m|, k(n*m)=|n| |m|$. 
\end{proof} 

\section{Additive number theory}

\paragraph{}
We illustrate the power of colored sticks, we use
a recursion theorem which was obtained in \cite{Knill1982}. It was found independently 
with the help of Cuisenaire material. It turns out that the mathematical theorem result was known 
since more than 100 years \cite{Ostmann1956,AndrewsPartitions} and probably was known 
to Euler already. Let $\sigma(n)$ denote the {\bf divisor $\sigma$-function} which gives the
sum of the proper divisors of an integer $n$. 
The {\bf partition number} $p(n)$ gives the number of different ways to write
$n$ as a sum of smaller positive integers where the order of the summands does not matter. 

\begin{thm}[1918] $p(n) = \sum_{i=1}^n \sigma(i) p(n-i)/n$ \end{thm}

\begin{proof} 
The arrangement of all partitions is a rectangle $P(n)$ of area $n p(n)$. 
Let $f \downarrow P(n)$ denote how many times the stick $f$ appears in the pile. 
One can now take from every row containing an $f$ one piece away to see 
$$  f \downarrow P(n) = f \downarrow P(n-f) \}| +p(n-f)  \;. $$
Repeating this gives
$$  f  \downarrow P(n) = f \downarrow P(n-f)  + f \downarrow P(n-2f) + P(n-2f) \; . $$
Therefore
$$  f \downarrow P(n)  =  \sum_{i=1} p(n-if) $$
if we assume $p(k)=0$ for negative $k$ and $p(0)=1$. 
Now, we sum up the area as a sum over all the sticks: 
\begin{eqnarray*} 
p(n) n  &=& [p(n-1) + p(n-2) + p(n-3) + p(n-4) + ... ]*1 \\
        &+& [p(n-2) + p(n-4) + p(n-6) + p(n-8) + ... ]*2 \\
        &+& [p(n-3) + p(n-6) + p(n-9) + p(n-12) + ...]*3 \\
        &+& [p(n-4) + p(n-8) + p(n-12)+ p(n-16) + ...]*4 \\
        &+& [p(n-n)                                  ]*n  \; .
\end{eqnarray*}
The theorem follows by noticing that $p(n-k)$ appears $\sigma(k)$ times in that array. 
\end{proof} 

\paragraph{}
The theorem can be seen also as a fixed point result $p = T(p)$ for a 
{\bf convolution average}
$$   \sigma * g(n) := \frac{1}{n} \sum_{k=1}^n \sigma(k) g(n-k) $$ 
with the divisor function $\sigma$. As convolution is commutative, we also have
$g *\sigma(n) = \frac{1}{n} \sum_{k=1}^n \sigma(n-k) g(k)$. 

\begin{figure}[!htpb]
\scalebox{0.18}{\includegraphics{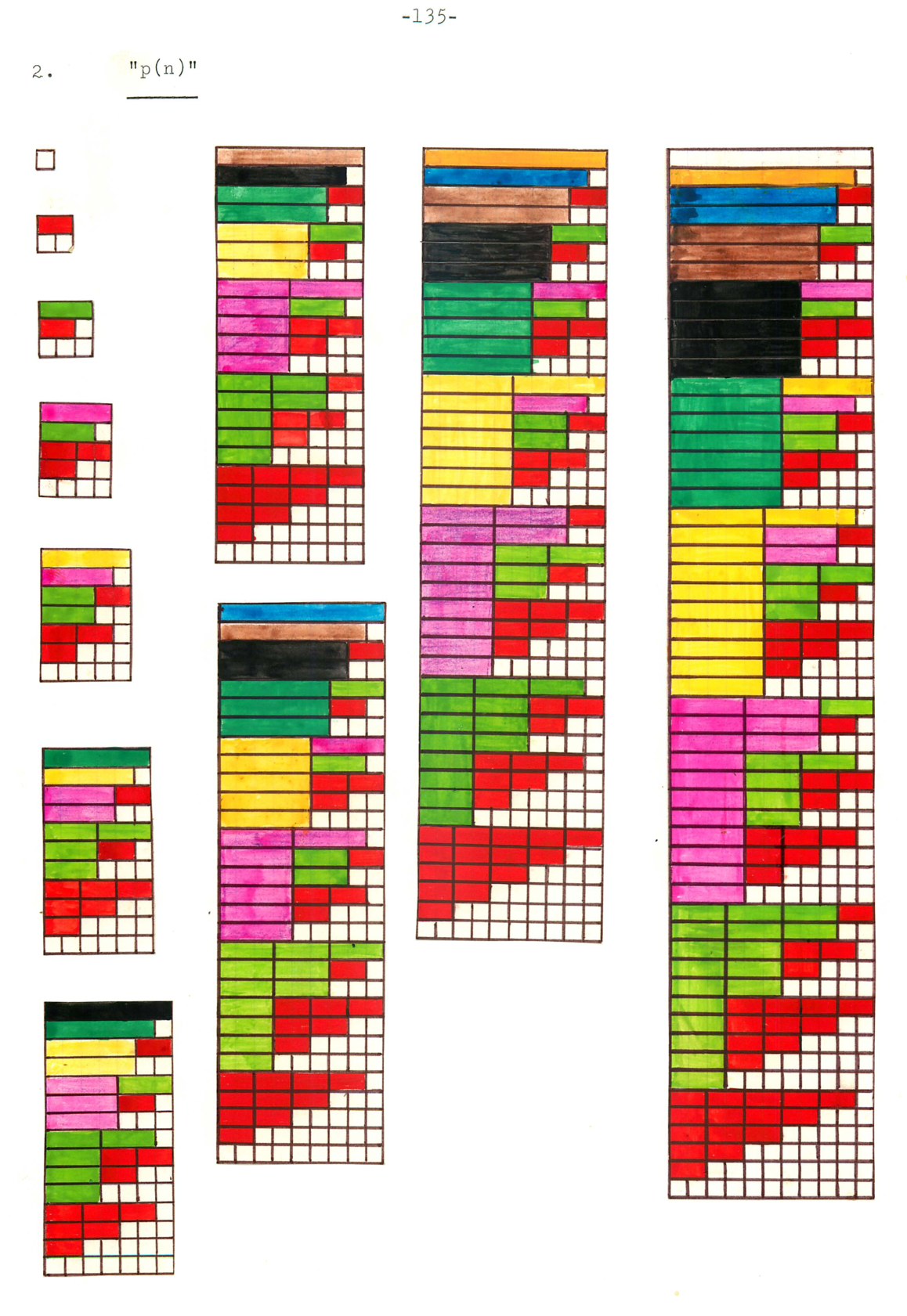}}
\scalebox{0.12}{\includegraphics{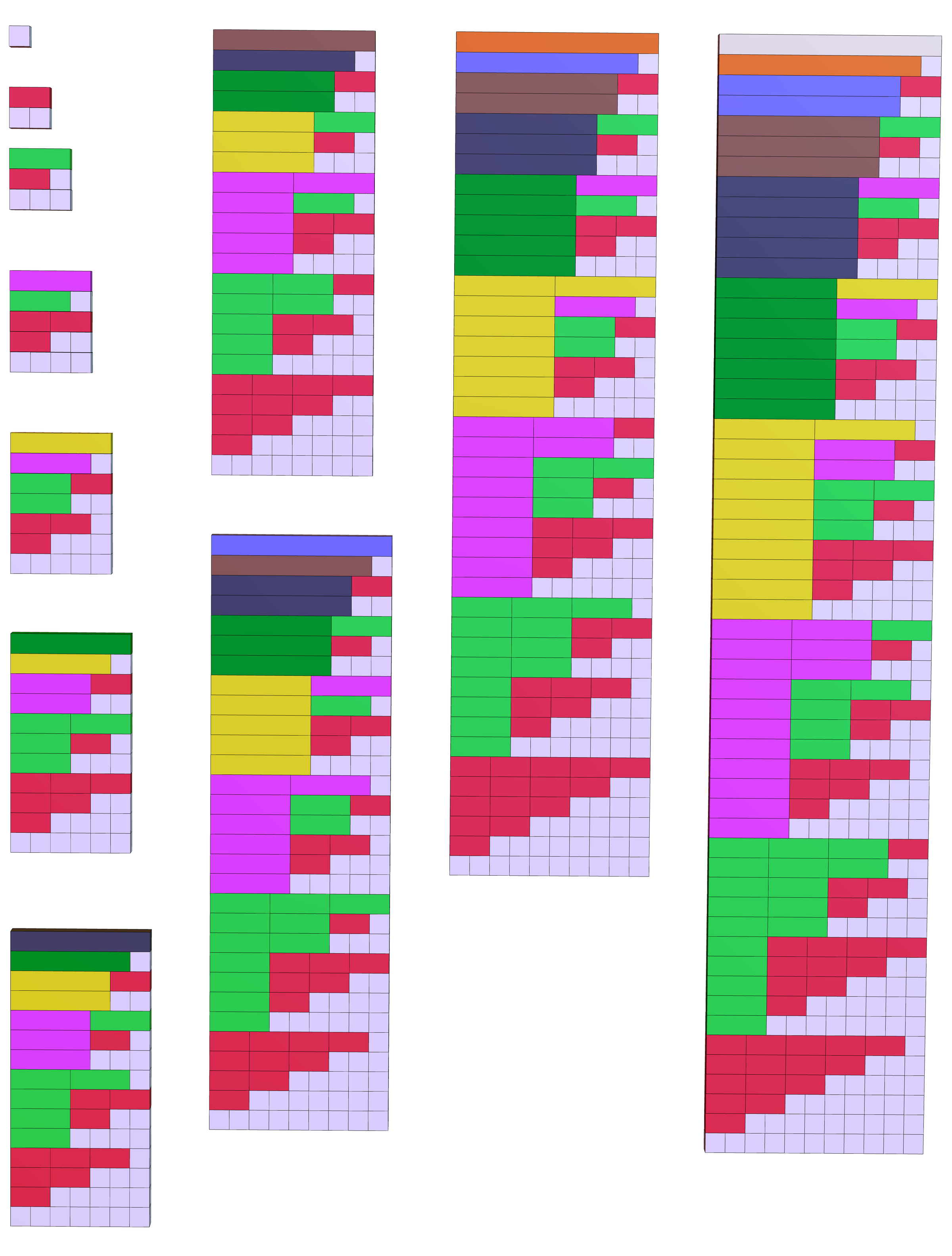}}
\label{cuisenaire}
\caption{
We see to the left first 10 partitions (see \cite{Knill1982} on page 135, 
hand drawn and pencil colorized or \cite{SJF1983}).
We have redrawn them in Mathematica in 2024. 
}
\end{figure}

\paragraph{}
The book title of the ATM resources \cite{Cane2017} shows similar arrangements for
partitions of numbers if the order matters. See also \cite{Knill1982} (page 133).
\cite{Cane2017} states {\it "After Gattegno Curriculum Graph 1974"}.
The idea of illustrating identities as such must have appeared first in 
\cite{GattegnoHoffman} page 5, interlocking staircases illustrating 
$10=1+9=2+8+ \dots=9+1=10$. We could not find any partition illustrations 
earlier than \cite{Knill1982}, also not in Gattegno's original books 
\cite{GattegnoMathematics}, nor in \cite{CuisenaireGattegno}. 

\paragraph{}
Computer algebra systems like Mathematica have the integer partition algorithm
already built in, but Euler's formula for the generating function
$$   \sum_{k=0}^{\infty} p(k) x^k = \prod_{j=1} \frac{1}{1-x^j} $$ 
\cite{Euler1753} quickly allows
to compute it from scratch. The third part of the following code uses
the implementation of the theorem. The built in Mathematica routine probably
uses the Euler Pentagonal number theorem.

\lstset{language=Mathematica} \lstset{frameround=fttt}
\begin{lstlisting}[frame=single]
Table[PartitionsP[n],{n,0,20}]
CoefficientList[Series[Product[1/(1-x^j),{j,20}],{x,0,20}],x]
F[0]:=1;F[n_]:=Sum[DivisorSigma[1,k]*PartitionsP[n-k]/n,{k,n}];
Table[F[n],{n,0,20}] 
\end{lstlisting}

\begin{figure}[!htpb]
\scalebox{0.95}{\includegraphics{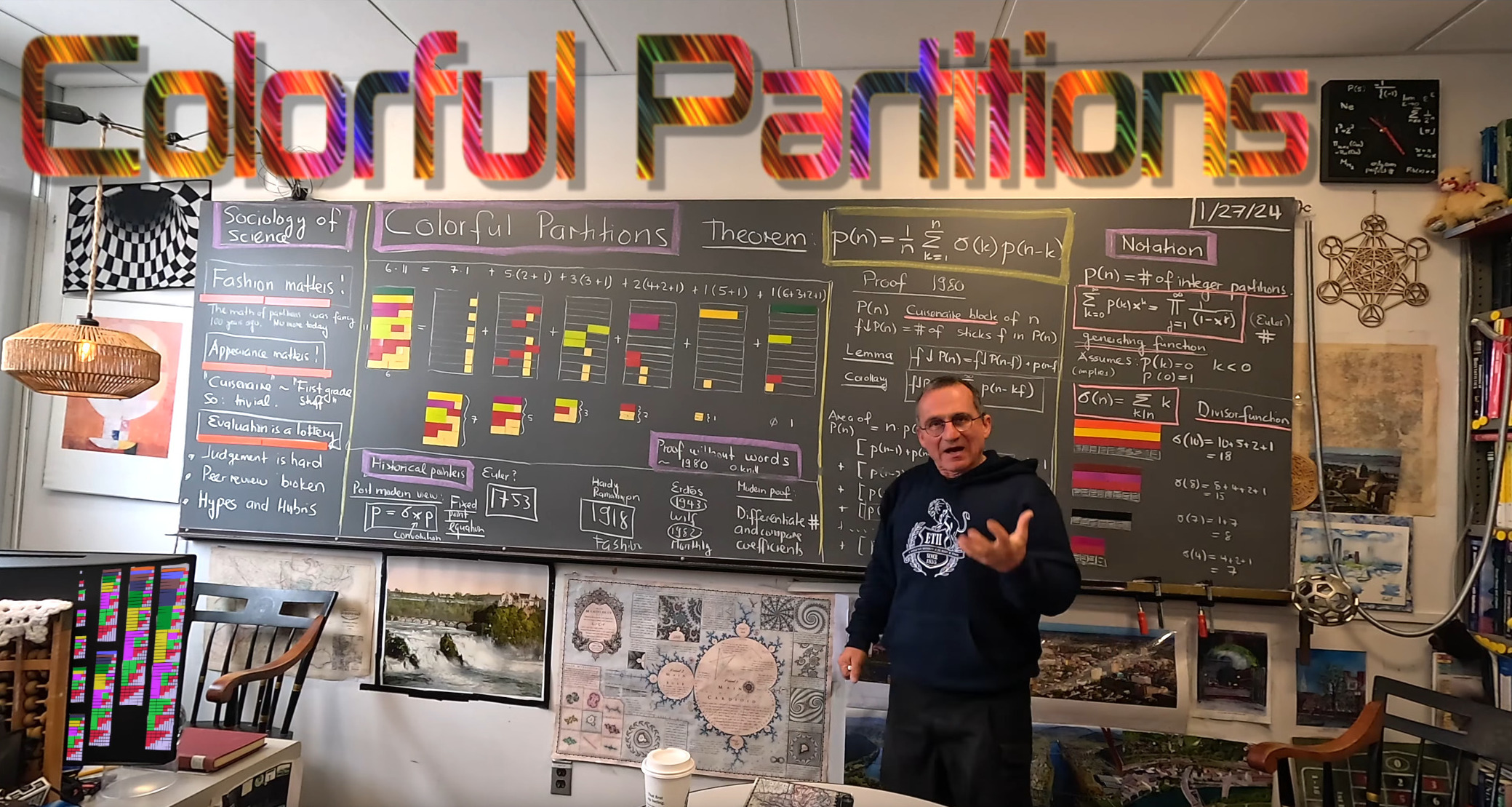}}
\label{youtube}
\caption{
The author presenting for a youtube recording on January 27th, 2024.
It illustrates the recursion formula visually, in an example
$6 * 11 = 7*1 + 5(2+1) + 3(3+1) + 2(4+2+1) + 1(5+1) + 1(6+3+2+1)$.
}
\end{figure}

\paragraph{}
As Euler noticed, the partition function and the divisor function satisfy
the same recursion. But the initial conditions are different:
$$  \sigma(n) = \sigma(n-1)+\sigma(n-2)-\sigma(n-5)-\sigma(n-7)+\sigma(n-12)+\sigma(n-15) + \cdots $$
$$  p(n) = p(n-1)+p(n-2)-p(n-5)-p(n-7)+p(n-12)+p(n-15) + \cdots  \; . $$
This is amazing because $\sigma(n)$ uses {\bf multiplicative properties}
of $n$ and $p(n)$ uses {\bf additive structures}. The initial 
conditions are different. Both assume $p(n)=\sigma(n)=0$ if $n<0$. 
The partition function starts with $(p(1),p(2))=(1,2)$ and takes $p(0)=1$.
The divisor function   starts with $(\sigma(1),\sigma(2))=(1,3)$ and uses 
$\sigma(n-n)=n$ in the recursion, when computing $p(n)$. 

\section{Dynamics}

\paragraph{}
How fast does the partition function $p(n)$ grow? The growth rates is known to be
exponential. The Hardy-Ramanujan asymptotic formula shows 
$p(n) \sim \exp(\pi \sqrt{2n/3})/(4n\sqrt{3})$. 
We repeat here a dynamical systems proof \cite{knillcotangent} 
of the following theorem: 

\begin{thm} $\limsup_n p(n)^{1/n} =1$ \end{thm}

\begin{proof}
{\bf Euler's pentagonal number theorem} tells that 
$$ Q(z)= \prod_{k=1}^{\infty} (1-z^k) = \sum_n a_n z^n \; , $$ 
where $a_n$ are the {\bf pentagonal numbers}. We also have
$$ P(z)=\sum_{n=1}^{\infty} p(n) z^n = \prod_{k=1}^{\infty} (1-z^k)^{-1} = \frac{1}{Q(z)}  \; . $$
{\bf Hadamard's lacunary theorem} \cite{kahane64} about power series 
$\sum_k c_k z^{n_k}$ shows that if $n_k+1/n_k \geq r>1$, then $r=1$ is a natural boundary.
This means that boundedness at one point $|z|=r$ would imply that the radius of convergence 
is larger than $r$. Define $g(\alpha) = (1-r e^{2\pi i\alpha})^{-1}$ so that
$f(\alpha) = \log(P(z)) = - \sum_{k=1}^{\infty} g(k \alpha)$.
For $r=|z|<1$, and $z=r \exp(2\pi i \alpha)$ with Diophantine
$\alpha$, the {\bf theorem of Gottschalk-Hedlund} \cite{KH} assures that $f$ is a 
{\bf coboundary}, meaning $f(x) = F(x+\alpha)-F(x)$, implying that the sum 
is bounded. Having so proven that the Taylor series has no singularity on
$|z|<1$, the partition function $p(n)$ satisfies $\limsup_n |p(n)|^{1/n} \geq 1$.
The other inequality $\limsup_n |p(n)|^{1/n} \leq 1$ implies together with
$p(n) \geq 1$ that $\sum_{n=1}^{\infty} p(n) z^n \geq \sum_{n=1}^{\infty} z^n$
so that the radius of convergence of the former is larger or equal than the radius
of convergence of the later, which is $1$. 
\end{proof} 

\paragraph{}
The proof illustrates how number theory, 
complex analysis, and dynamical systems theory and
Diophantine notions can come together. Such connections also appeared when studying
Dirichlet series with almost periodic coefficients \cite{KL}. 

\section{Psychology}

\paragraph{}
How do human learn, teach and do research? How do they gain intuition about mathematical 
objects. Psychological aspects of teaching have been looked at in research \cite{Hadamard1954}
or education \cite{Norton1992}. Hadamard stresses that visual thinking is key for intuition. 
Norton suggests that a number is defined through 
our own actions and mentions in the chapter "What are numbers" 
cites historical definitions of numbers like Euclid's notion 
{\it "A number is a multitude composed of units."} appearing in book VII.
The development of numbers is also linked to language and notation
\cite{EnlighteningSymbols}. A more topological approach to numbers are quipu, talking knots
\cite{Locke1912,AscherAscher}.

\paragraph{}
In \cite{GraphsGroupsGeometry}, the question was raised 
whether one can challenge even the beginning of counting. Kronecker is reported to have told 
the famous {\bf "God made the integers; all else is the work of man"} 
\cite{HawkingGodCreatedIntegers,Bell}. But are the integers really the most natural way for
counting?

\paragraph{}
In \cite{GraphsGroupsGeometry} we showed that the {\bf infinite dihedral group} $Z_2*Z_2$ generated 
by two involutions is more "natural" than $\mathbb{Z}$. This was neither opinion nor
a philosophical question; it emerged from a definition: a group is ``natural" if there exists
a metric space $X$ such that there is only one way up to isomorphism to put a group structure 
on $X$ for which all group operations are isometries. The infinite dihedral group is natural. 
The integers are not. There are also more down-to-earth reasons why the dihedral group is 
natural. Maybe the most important one is that there is no Grothendieck group 
completion from the monoid $(\mathbb{N},+,0)$ to $the group (\mathbb{Z},_,0)$ needed. 
From the natural numbers, a young student first has to be convinced about 
negative numbers using analogies like money debt or negative temperatures. In the dihedral case 
$1=ab$ and $-1=ba$. A completely different approach sees numbers as games \cite{numbersgames}. 
Donald Knuth visualized it in the novel numbers and games \cite{knuthnumbers}. 

\paragraph{}
While commutative structures have been challenged since
the emergence of quantum mechanics and under the umbrella of non-commutative mathematics \cite{Connes},
the commutativity axiom for the additive structure of the integers has never been 
put up to debate. But as argued in \cite{GraphsGroupsGeometry}, there are reasons to 
consider a group generated by two involutions more fundamental than a group by a 
single generator $1$. The dihedral approach very much is in the philosophy of reflection 
geometry \cite{Bachmann1959,JegerTransformationGeometry} where reflections are used to generate
the non-Abelian Euclidean group $\mathbb{R}^n \rtimes O(n)$ consisting of translations and orthogonal
transformations. Reflections at affine planes are all involutions. {\bf dihedral counting} 
enters geometry if we look at the group generated by the exterior derivative $d$ and its dual $d^*$. 
We have $d^2=0$ and $(d^*)^2=0$.  

\paragraph{}
Also notation and implementation is important. Good notation pays off in the long term 
by allowing to think more effectively. 
Leibniz's notations and nomenclature in calculus for example was much persuasive than Newton's 
notation. In the case of manipulatives, the analog of notation is how to use {\bf material}, 
{\bf shape}  or {\bf color} and {\bf labels} implemented in texture. Katherine Stern 
used colored cubes around a similar time than Cuisenaire. But the colors did not carry 
intuitive meaning. Other manipulatives had writings on them, making them less elegant.

\paragraph{}
In the case of Cuisinaire sticks, the colors have been carefully chosen.   
Cuisenaire thought a great deal about the psychology of numbers. He tried
to use colors which match division properties. Of course this can not be
pulled through in detail but the choice  $red=2 \to crimson=4 \to brown=8$ 
or $yellow=5 \to orange =10$    or $green=3 \to darkgreen=6 \to blue = 9$ 
and then black=7 shows taste. Each of the 4 smallest {\bf primes} $2,3,5,7$ 
in the range of 1 to 10  used. Multiples then are using similar color components.

\begin{itemize}
\item Each of the primes 2,3,5,7 starts a new color family.
\item The even numbers 2,4,6,8 have a red-purple components, 
\item The multiples of three, 3,6,9, have a green-bluish side,
\item The multiples of 5, given by 5 and 10 have a yellow-orange part.
\end{itemize}

\paragraph{}
A young mathematics student playing with these sticks can gain so additional 
understanding of numbers. There is a geometric link in the form of length. 
There are also color impressions. Also music exposure can provide {\bf auditory link} 
to numbers. The do-re-mi-fa-so-la-ti-do already introduces a modular arithmetic, in this case 
$Z_8$, the smallest three dimensional vector space. In music, there are 
other dimensions like amplitude or modulation. Music implements number systems like
the modular arithmetic on scales. But is is much richer. 

\section{History}

\paragraph{}
We take the opportunity to make a list of manipulative material used by educators. 
We see that the first attempts to build manipulatives were done at the end of the
18th century by educators like Friedrich Froebel or Ernst Tillich. Froebel is known
primarily as a philosopher from the book \cite{Froebel} who believed in the 
fundamental principle that there must be an inner connection between the pupil's mind
and the objects which are studied. The contemporary pedagogue and education reformer
Johann Heinrich Pestalozzi (1746-1827) is considered the first modern educator.
He is known for the picture that training children should involve 
{\bf "heart, mind and body"} and that nature was the best way to teach. 

\paragraph{}
As for the development for mathematical toolboxes or manipulative, see \cite{Jobbe-Duval}.
Jobbe-Duval gives as a reference about the first appearance (1947) of the Cuisenaire
rods an article \cite{BulletinCuisenaire1964}. Wikipedia dates the first appearance of
the sticks to 1945. We cite from this Wikipedia entry
{\it "In 1945, following many years of research and 
experimentation, Cuisenaire created a game consisting of coloured cardboard 
strips of various lengths that he used to teach mathematics to young children."}
We can see simple tasks like counting trees or estimating the number of trees in a forest
or to watch for Fibonacci patterns in pine cones or sunflowers or see fractions in the 
harmonies of music as examples of such teaching methodologies. 

\def\arraystretch{1.5}
\begin{center} 
\begin{tabular}{|l|l|l|}  \hline
Author               & Year       & Method    \\ \hline \hline
Friedrich Froebel    & (1782-1852)& wooden building blocks \\ \hline
Ernst Tillich        & (1780-1807)& Rechenkasten \\ \hline
Catherine Stern      & (1894-1973)& blocks \\ \hline
Maria Montessori     & (1870-1952)& barres numeriques rouges es blues \\  \hline
Georges Cuisenaire   & (1891-1975)& wooden block in 1951  \\ \hline
Mina Audemars        & (1883-1971)& 66 blocks \\ \hline
Louise Lafandel      & (1872-1971)& 66 blocks  \\ \hline
Setton Pollock's     & (1910-1983)& Color Factor system 1960ies \\ \hline
Artur Kern           & (1902-1988)& Rechenkasten \\ \hline
Paulette Calcia      & (1904-1985)& Buchettes'dor \\ \hline
Maria-Antonia Canals & (1930-)    & Reglets \\ \hline
Adam Langer          & (1836-1919)& Rechenkasten \\ \hline
Remi Brissiaud       & (1949-)    & Noumes   \\ \hline
\end{tabular}
\end{center} 

\begin{figure}[!htpb]
\scalebox{0.62}{\includegraphics{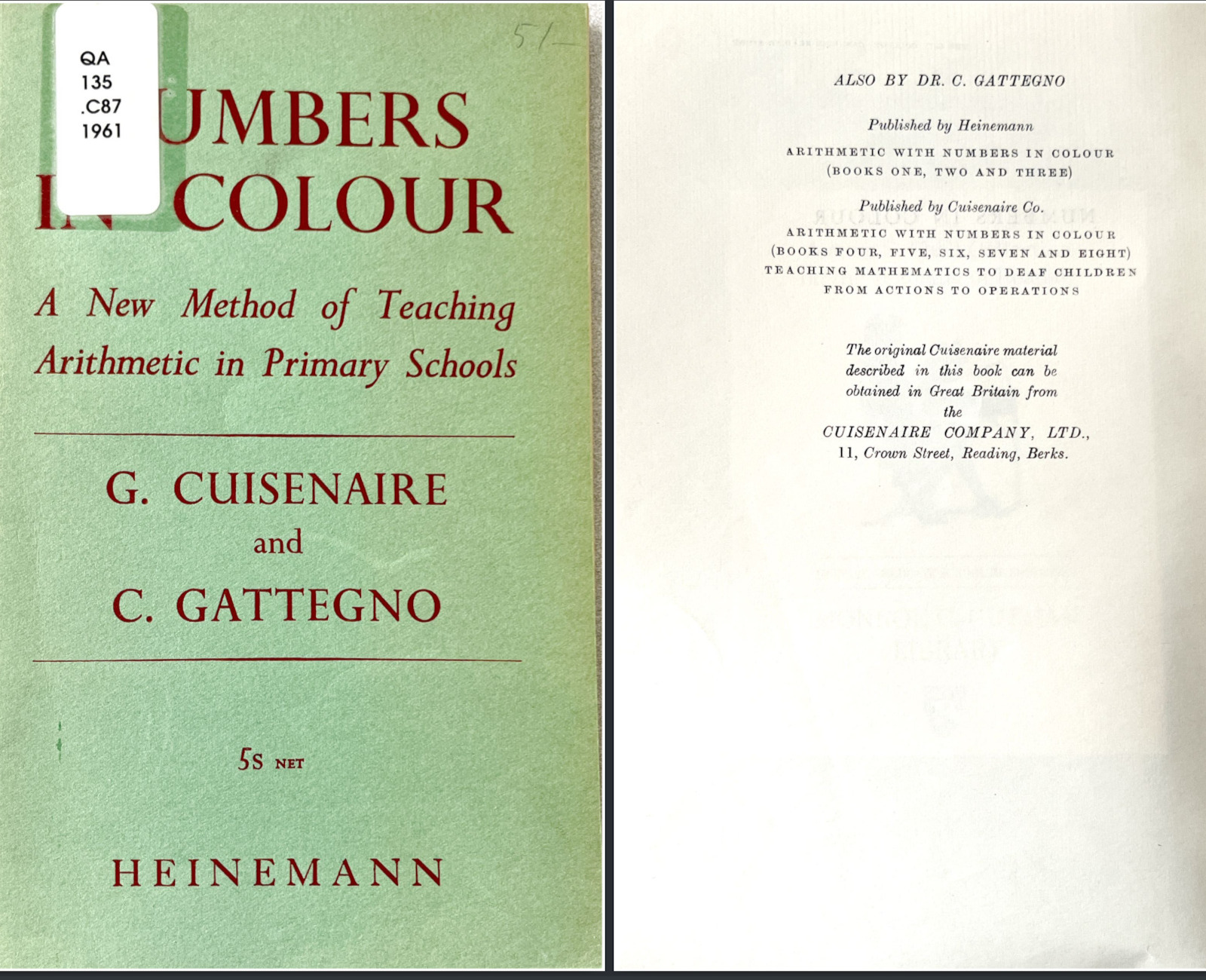}}
\scalebox{0.62}{\includegraphics{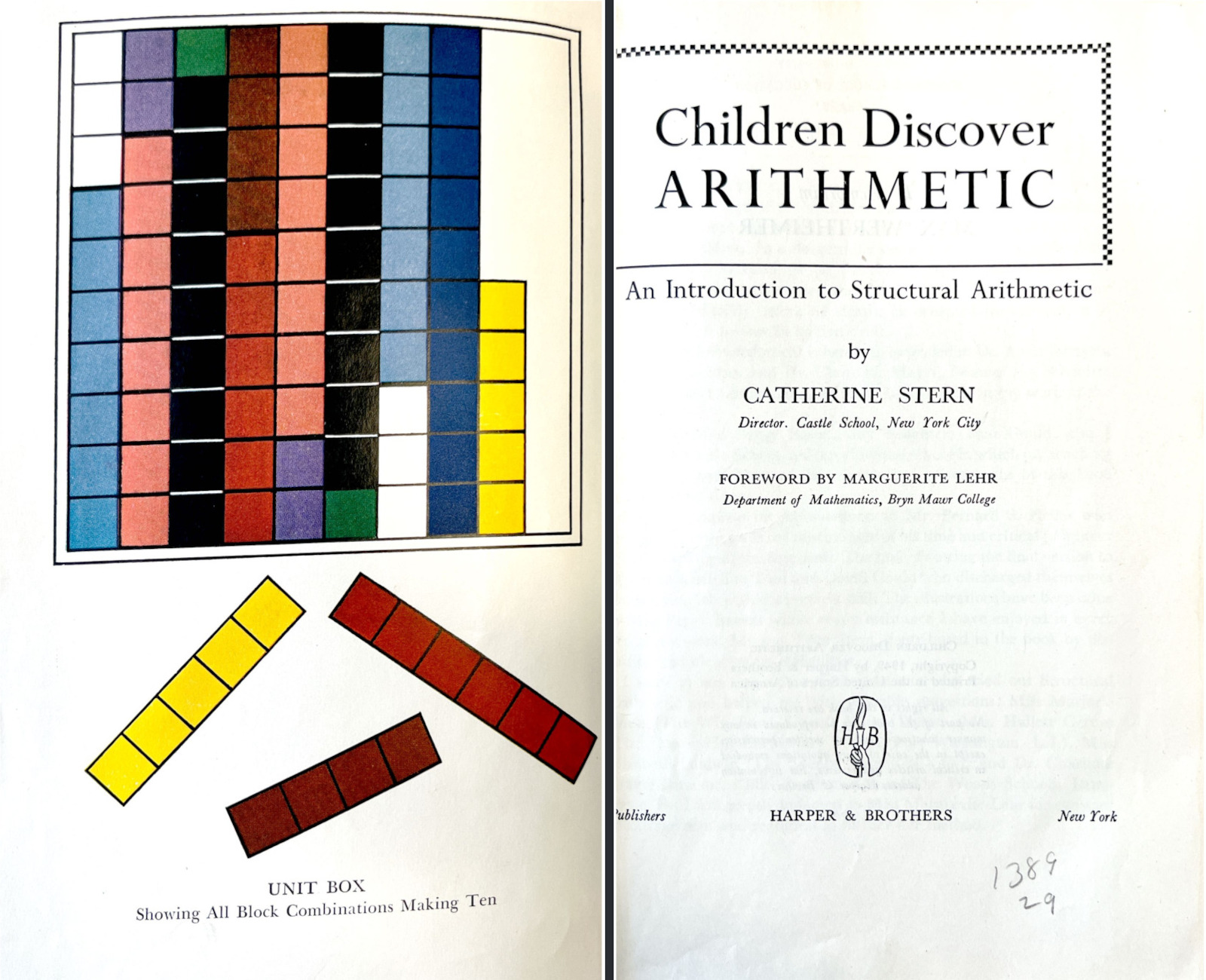}}
\label{Movie}
\caption{
The book cover of Stern \cite{Stern1949} from 1949 
and the book cover of Cuisenaire-Gattegno \cite{CuisenaireGattegno} from 1954.
}
\end{figure}

\paragraph{}
The first colored Cuisenaire were known since 1947 and were
published in 1951. It appears as if the approach of Catherine Stern was independent
from Cuisenaire's and that Stern was influenced by Montessori.
The ``Stern Math" approach is also still used today. On the website of the organization
is a quote of Einstein saying 
{\it "I believe that her (Catherine Stern's) idea is sound would be of real value
in the teaching of the elements of arithmetic".}

\begin{figure}[!htpb]
\scalebox{0.32}{\includegraphics{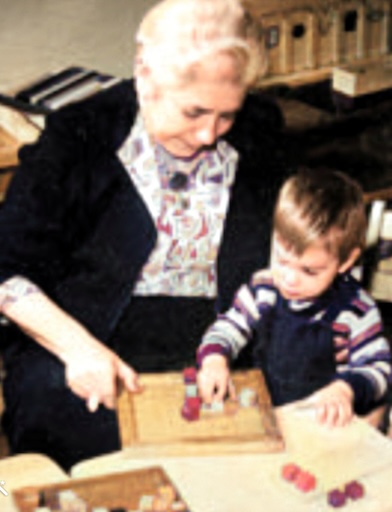}}
\scalebox{0.32}{\includegraphics{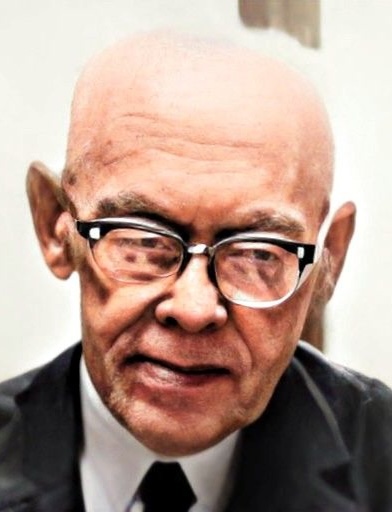}}
\label{People}
\caption{
We see the two pioneers George Cuisenaire  and Catherine Stern. 
The my-heritage colorized photo shows her in 1950 with her grandson Fred.
According to \cite{Sternmath}, Catherine stern was influenced by 
Maria Montesori's ideas and presented arithmetic manipulatives in 1934 during a
conference of the Swiss Kindergarden Association. She immigrated to the US in 1938. 
The photo of George Cuisenaire is from \cite{Cutler2023}, also colorized 
with my-heritage.
}
\end{figure}

\paragraph{}
Research in partitions was in full swing at the time of Ramanujan
who would mention partitions already in his first letter to 
Hardy in January 1913.

\begin{figure}[!htpb]
\scalebox{0.40}{\includegraphics{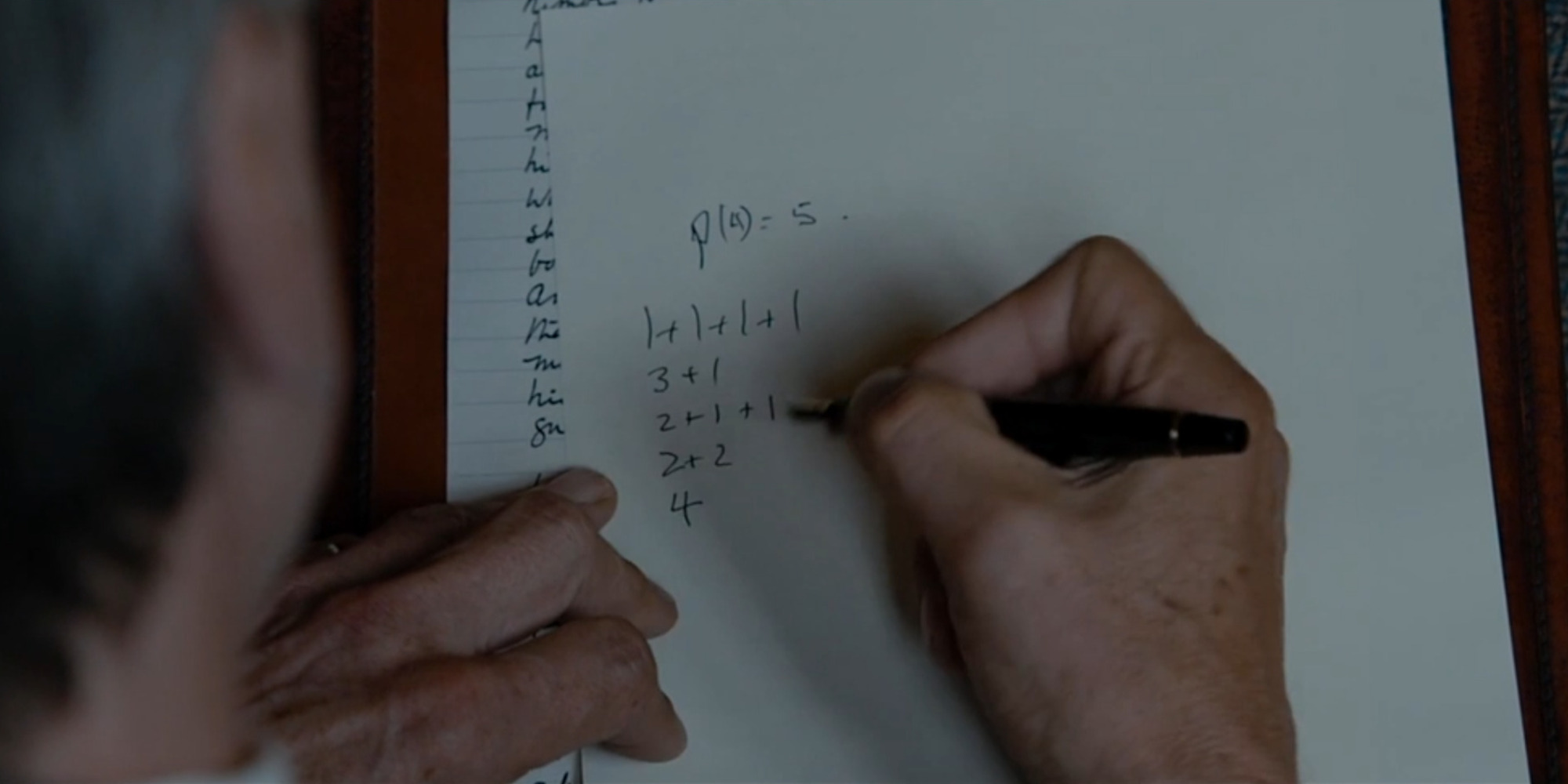}}
\scalebox{0.40}{\includegraphics{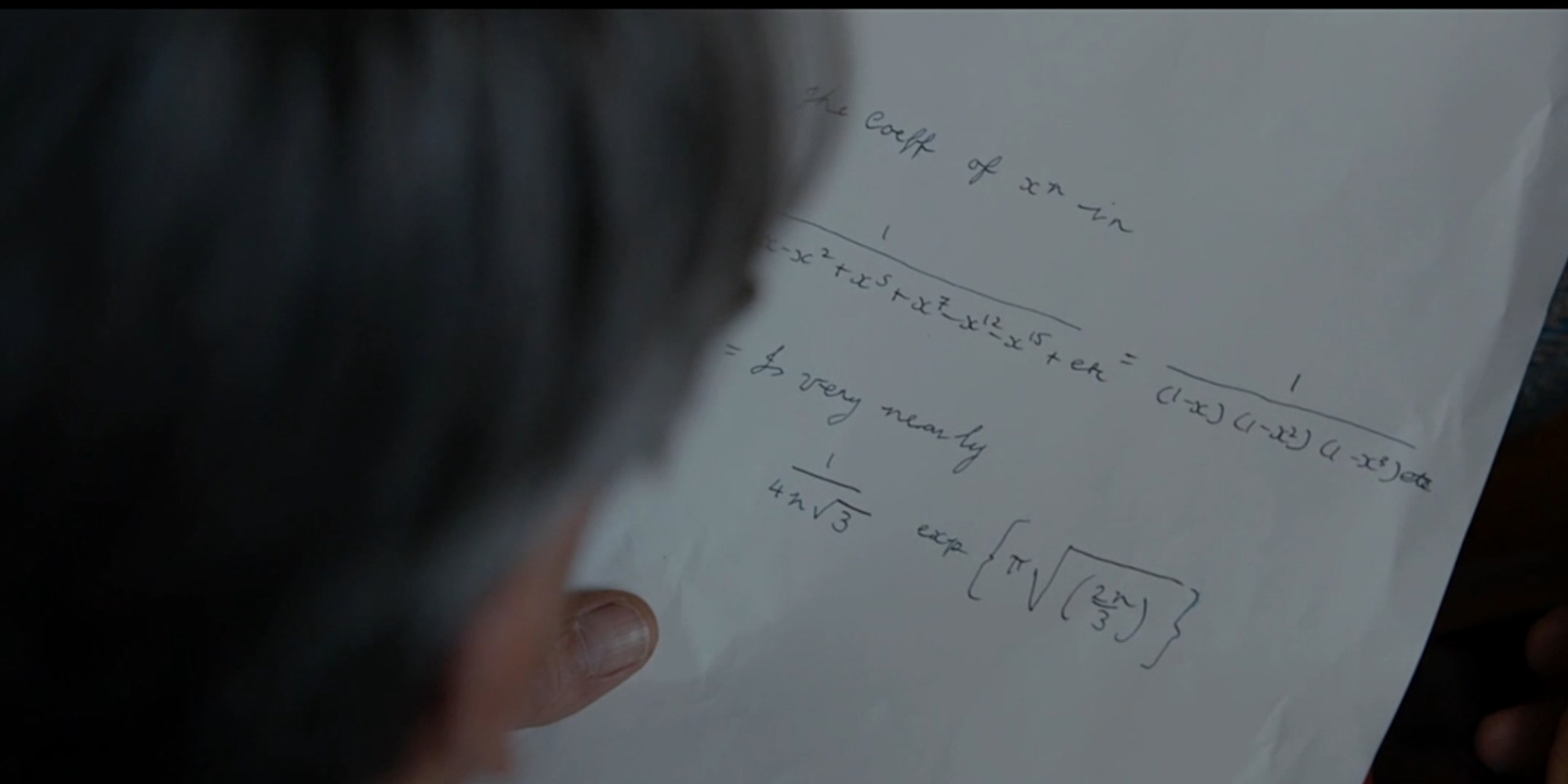}}
\label{Movie}
\caption{
Hardy (played in the movie by Jeremy Irons) explains the asymptotic 
Hardy-Ramanujan formula for the partition function in the 2015 film: 
``The man who knew infinity".
}
\end{figure}

\paragraph{}
Partitions are mentioned in the book about Ramanujan 
\cite{Kanigel} (pages 246-250). The movie adaptation of the book 
makes partitions to a small drama in the form of a competition. 
Hardy explains the problem in the movie similarly as in the book with 
a small example $p(4)$.  One can the see in a scene 
a manuscript showing the Hardy-Ramanujan formula
$p(n) \sim R(n)=1/(4n \sqrt{3}) \exp(\pi \sqrt{2n/3})$ meaning 
$p(n)/R(n) \to 1$. 

\nopagebreak 

\section{Creativity}

\paragraph{}
Partitions appear when defining discrete sub-manifolds of given manifolds 
\cite{ManifoldsPartitions}. While working on this emerged the picture that
the addition of two partitions $n=(n_1 \dots, n_k)$, $m=(m_1, \dots, m_l)$,
in the form of concatenation $n+m = (n_1, \dots, n_k,m_1, \dots, m_l)$
and the product by foiling out:
$n*m = (n_1 m_1, n_2 m_1, \dots, n_k m_1, \dots, n_1 m_k \dots, n_k m_l)$.
For example $(3,4) + (1,1,2) = (3,4,1,1,2)$ and $(3,4)*(1,1,2) = (3,3,6,4,4,8)$.
It leads to the idea to associate to a partition $n=(n_1,\dots, n_k)$ the
$k$-partite graph $K_{n_1,\dots, n_k}$. The partition $5=K_{1,1,1,1,1}$ for example 
is $K_5$ and $6==K_{3,3}$ is the utility graph. 
The addition is the Zykov join and the multiplication is the Sabidussi multiplication.
This ring structure on multi-partite graph is dual to the Shannon ring structure, where
addition is disjoint union and multiplication is Shannon multiplication.

\paragraph{}
Let me add a personal remark about education, which is very much related
to my early exposure to Cuisenaire sticks.
In the last 40 years, one can observe a {\bf standardization tendency} 
in K-12 education. The reason is a reflex to {\bf minimize risk}
from the administration side and to make sure that basic goals are achieved. 
But chaining of teachers to fixed curricula indicates 
an increasing lack of trust for individual initiative. It is a vicious
circle because such a system discourages creativity, it also is 
also less likely attracts creative minds into the profession. 
It leads to a de-valuation of public schools, the emergence of 
after-school programs or private schools, which are not chained to such 
administrative handcuffs and where teachers can fully show their potential. 

\paragraph{}
In the primary school up to high school of 50 years ago, our teachers could get us out into the open,
mix art or sport with mathematics without administrative oversight.  
I myself enjoyed orienteering competitions in larger forest with professional
orienteering maps and compass learning about direction, level curves, or 
excursions to collect plants in order to build our own botany book,
or  collect rock samples, climb on the cold volcanoes in the southern 
Germany, or visit the Danube sinkhole. In high school chemistry, we were for once 
just given a powder and access to the generous school chemistry lab. 
The task was to use whatever tool we could find and gets hand on 
to figure out what the powder was. There were no worksheets,
just expert guidance and supervision by teacher and 
an expert lab assistant helping to keep us safe. I still remember 
the excitement after finding out that the powder given to us was ``Aspirin".

\paragraph{}
The freedom which these teachers had appears more exceptional today. 
It has become the privilege of private institutions. The times of the ``Black Mountain college" 
\cite{Schoon,DiazExperimenters}, where rival methodologies were explored and experimental forms were practiced,
are long over. (The college is featured in the 2013 novel "The longest ride" by Nickolas Sparks, and also
made it into a movie in 2015, where the free teaching style at the college is very well visible.)

\paragraph{}
By the way, the mathematician Max Dehn, who was a student of Hilbert, had taught from 1945-1952 
at the Black Mountain college and is buried there \cite{Sher}. 
The Wikipedia article about him (which follows \cite{Sher}) states {\it He enjoyed the forested mountains found 
in Black Mountain, and would often hold class in the woods, giving lectures during hikes.
His lectures frequently drifted off topic on tangents about philosophy, the arts, and nature and their 
connection to mathematics. He and his wife took part in community meetings and often ate in the dining room. 
They also regularly had long breakfasts with Buckminster Fuller and his wife.} As as source,
\cite{Peifer2011} is given. 

\paragraph{}
The following quotation given in \cite{Sher} about math education is remarkable:
{\it All instruction, including mathematics, should be an end in itself, not, in the 
first place, means. Therefore in mathematics one should be so simple and elementary 
as possible so that the structure of mathematical thinking becomes clear. If this 
is achieved then the student will have less difficulty to see mathematical
structures in physical phenomena or even in phenomena of biology and social life.}
And from \cite{Peifer2011}: {\it As Dehn was famous for his way of picturing 
abstract concepts, many of his great mathematical ideas can be intuitively 
understood through pictures.}

\begin{figure}[!htpb]
\scalebox{0.12}{\includegraphics{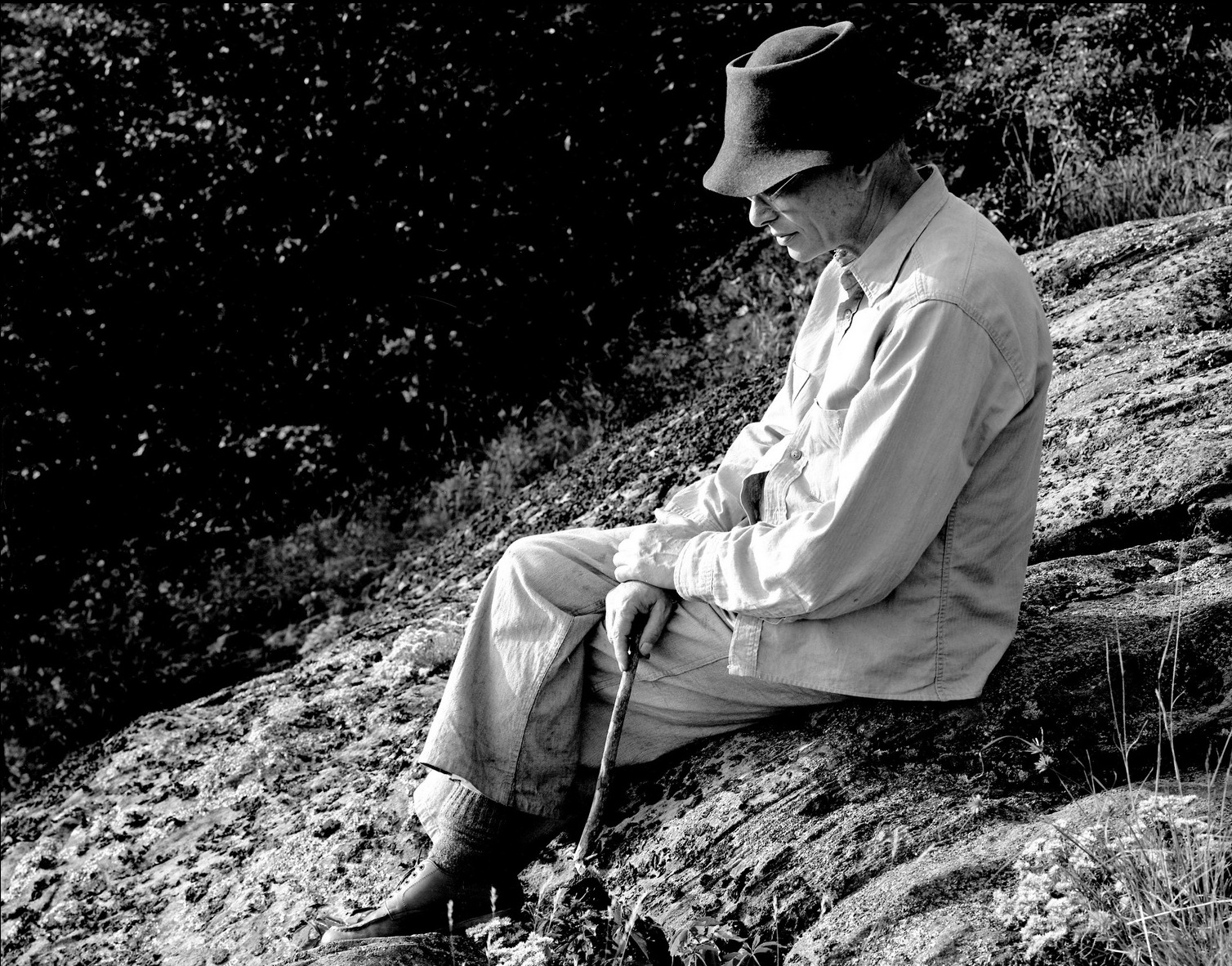}}
\label{cuisenaire} 
\caption{
This photo credited to the photography critic {\bf Nancy Newhall} (1908-1974) 
is titled ``Max Dehn at the Black Mountain College" is dated around  1947.
The gelatin silver negative had been gifted by {\bf Christi Newhall} 
to the Eastman Museum, where one can find a high resolution version. 
Nancy Newhall visited together with her husband Beaumont Newhall the Black Mountain College 
for three summers in the years 1946-1949. The Newhalls photographed during that time 
the campus and its people.
}
\end{figure}

\paragraph{}
Also in higher education, the 
administrative cages have been rolled out more and more.
We could still develop my calculus courses independently without oversight 10 years ago. 
Of course, there were course goals and coordination between different courses but the 
details, the homework or lesson plans were not prescribed. 
Today, we live in a time, with much more constraints. It appears safer like that but
a teaching profession that lacks initiative and originality will be 
replaced by technology much faster than anticipated. 

\section{More personal remarks}

\paragraph{}
I myself got introduced to numbers using the Cuisenaire material \cite{CuisenaireGattegno}.
This had happened on the personal initiative by my first
grade teacher {\bf Edwin Ilg} \cite{Schulblatt1968}
who worked at a time, when teachers could shape their curricula.
We must maybe stress that this happened at an ``ordinary" public school 
in Uhwiesen, a village near Schaffhausen in Switzerland.
I still own the same Cuisenaire sticks from first grade. 

\paragraph{}
The independent research done in high-school had led to the theorem mentioned above and others like the 
Pentagonal number theorem (which I however could not prove then). 
The topic of partitions came up during the winter 2023/2024 in the context of geometry, when looking
again at the {\bf discrete Sard theorem}. That theorem has classically has been developed 
by Morse and Sard \cite{Morse1939,Sard1942}. It is pretty remarkable that one can do a similar result
within finite mathematics: first define carefully what a $d$-manifold is (a finite simple graph for which 
every unit sphere is a $(d-1)i$-sphere, where a $d$-sphere is a $d$-manifold that becomes contractible if
a vertex is taken away. In 2015, I noticed that one can define hyper-surfaces in such a manifold by looking
at all simplices on which the function changes sign \cite{KnillSard}. The only requirement is that the function
$f$ is locally injective (a coloring).  During Fall 2023 this was then generalized to higher dimensions
\cite{ManifoldsPartitions}

\paragraph{}
Again related to teaching, when leading a day long workshop in 2017 on the ``Adventures in algebra",
\cite{AdventuresAlgebra}, I wanted to make the point of "being creative" and spent the 
winter 2016/2017 preparing for that January 2017 event in McAllen in the Rio Grande Valley. 
During that exploration time, I wondered whether there is a multiplication which fits the join 
operation in graph theory. A longer search revealed the Sabidussi multiplication (as I later would find out). 
It is frequent in research that one rediscovers something that has already been found. 
But as with the high school research in number theory, it does not diminish 
the joy of finding it. It is also fun to find something that has already been known. 

\begin{figure}[!htpb]
\scalebox{0.07}{\includegraphics{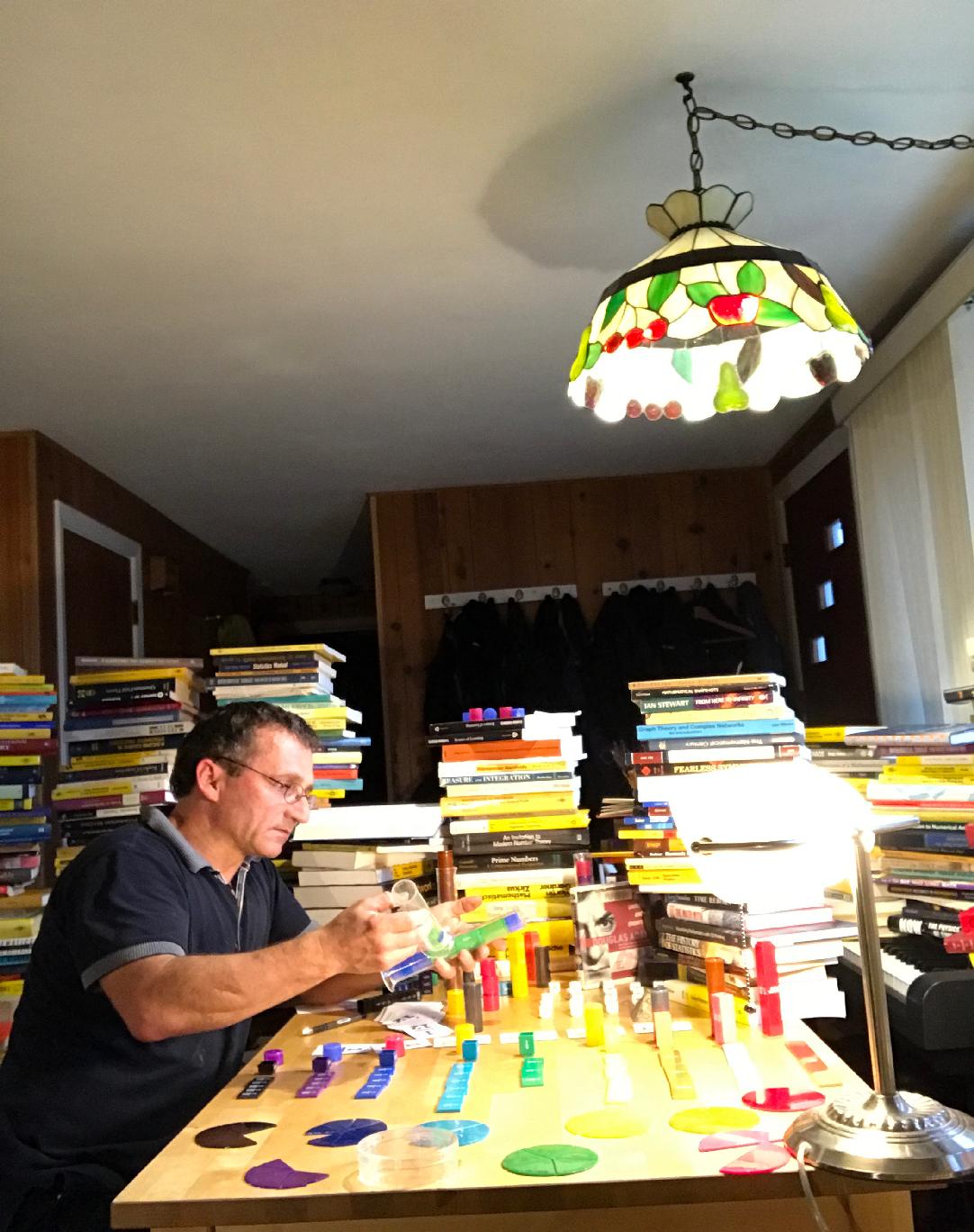}}
\scalebox{0.07}{\includegraphics{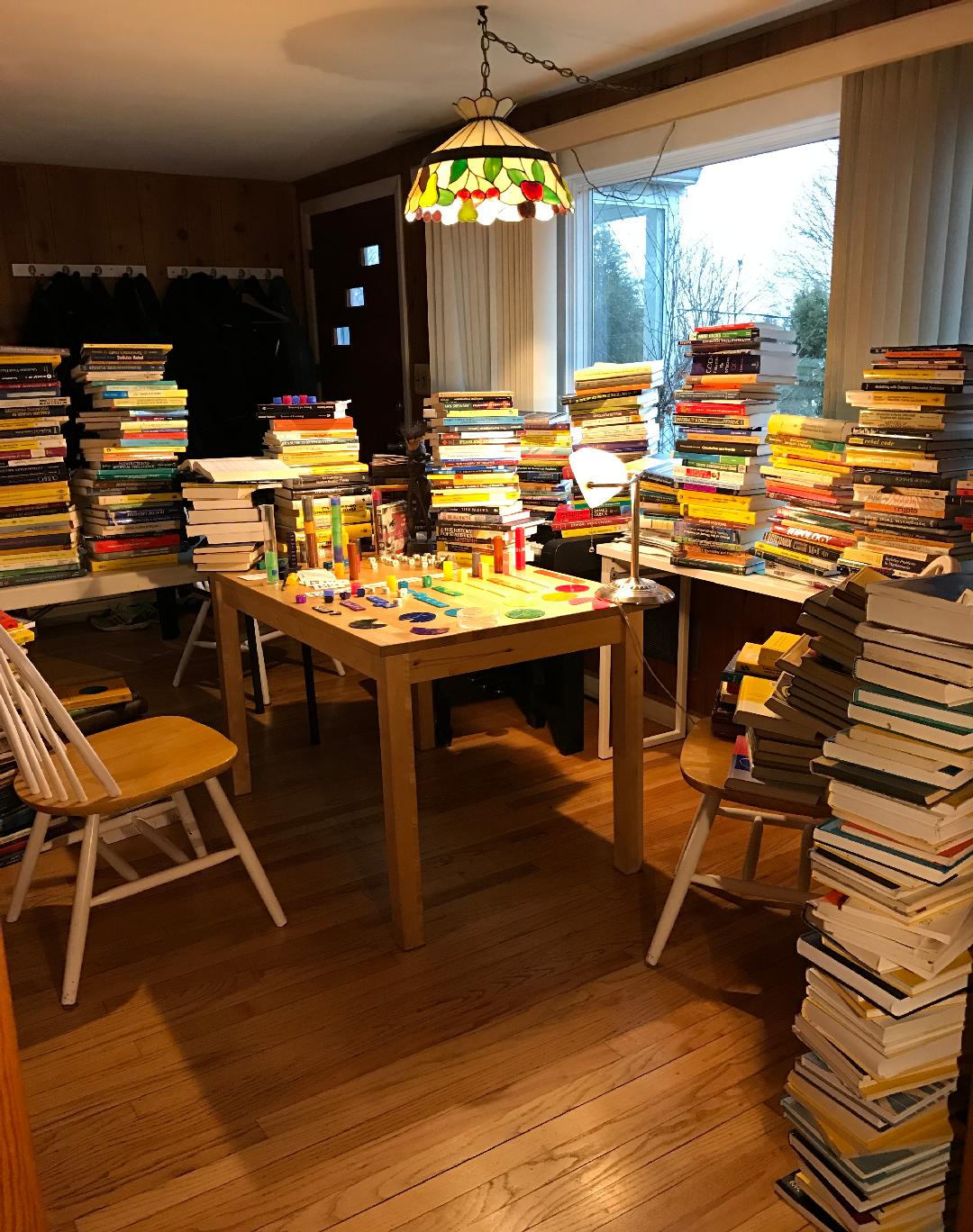}}
\scalebox{0.07}{\includegraphics{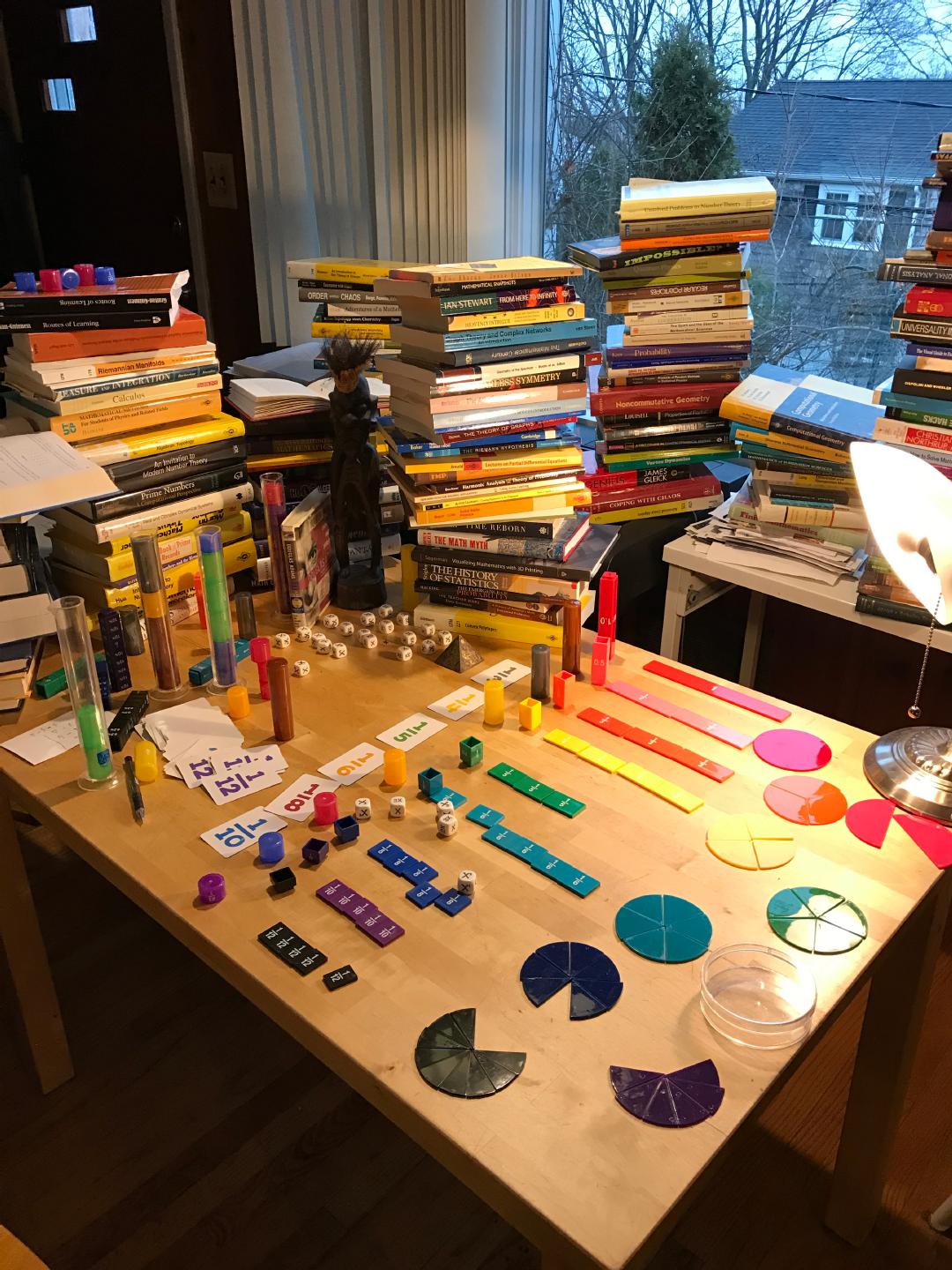}}
\scalebox{0.07}{\includegraphics{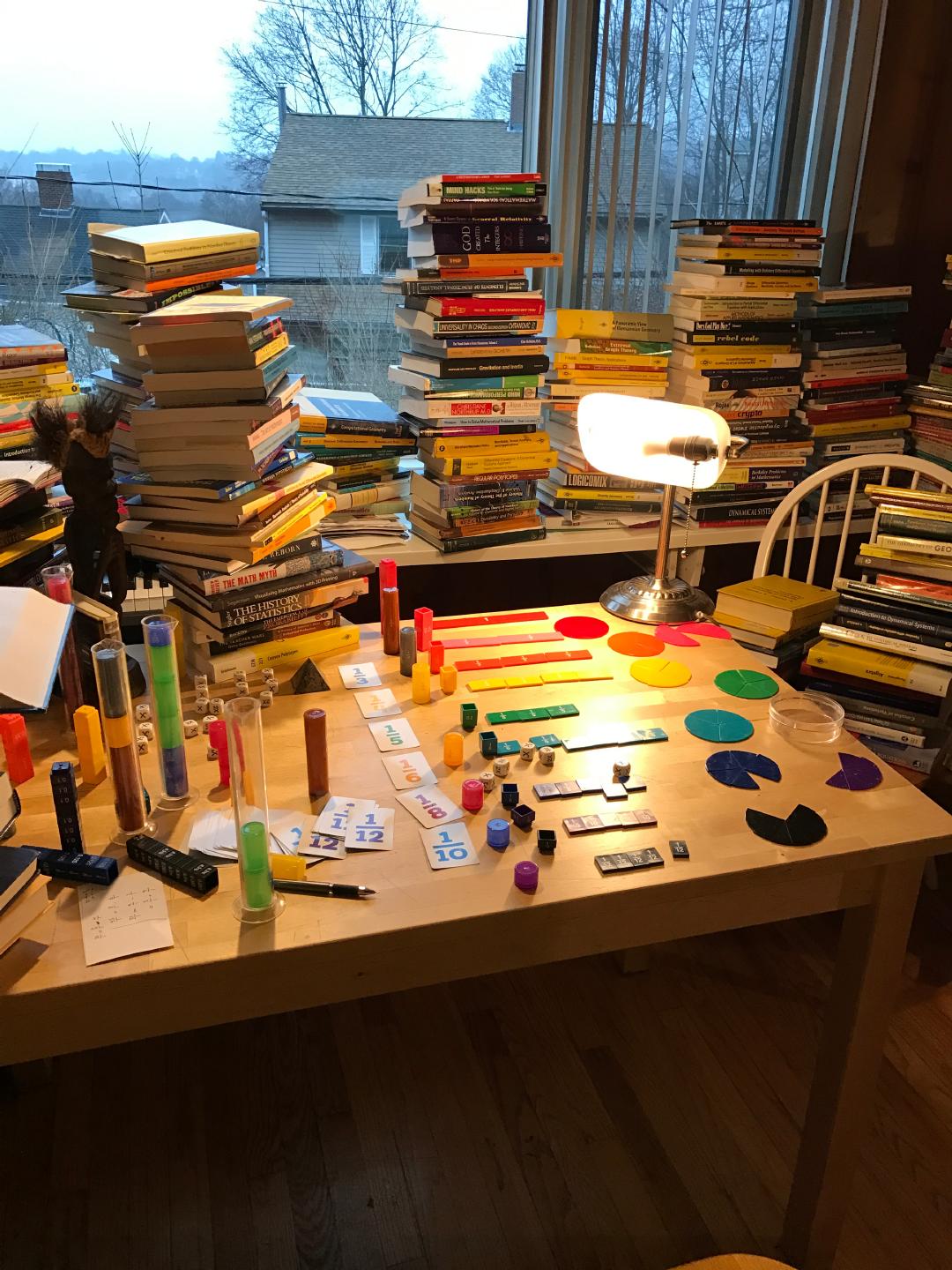}}
\scalebox{0.07}{\includegraphics{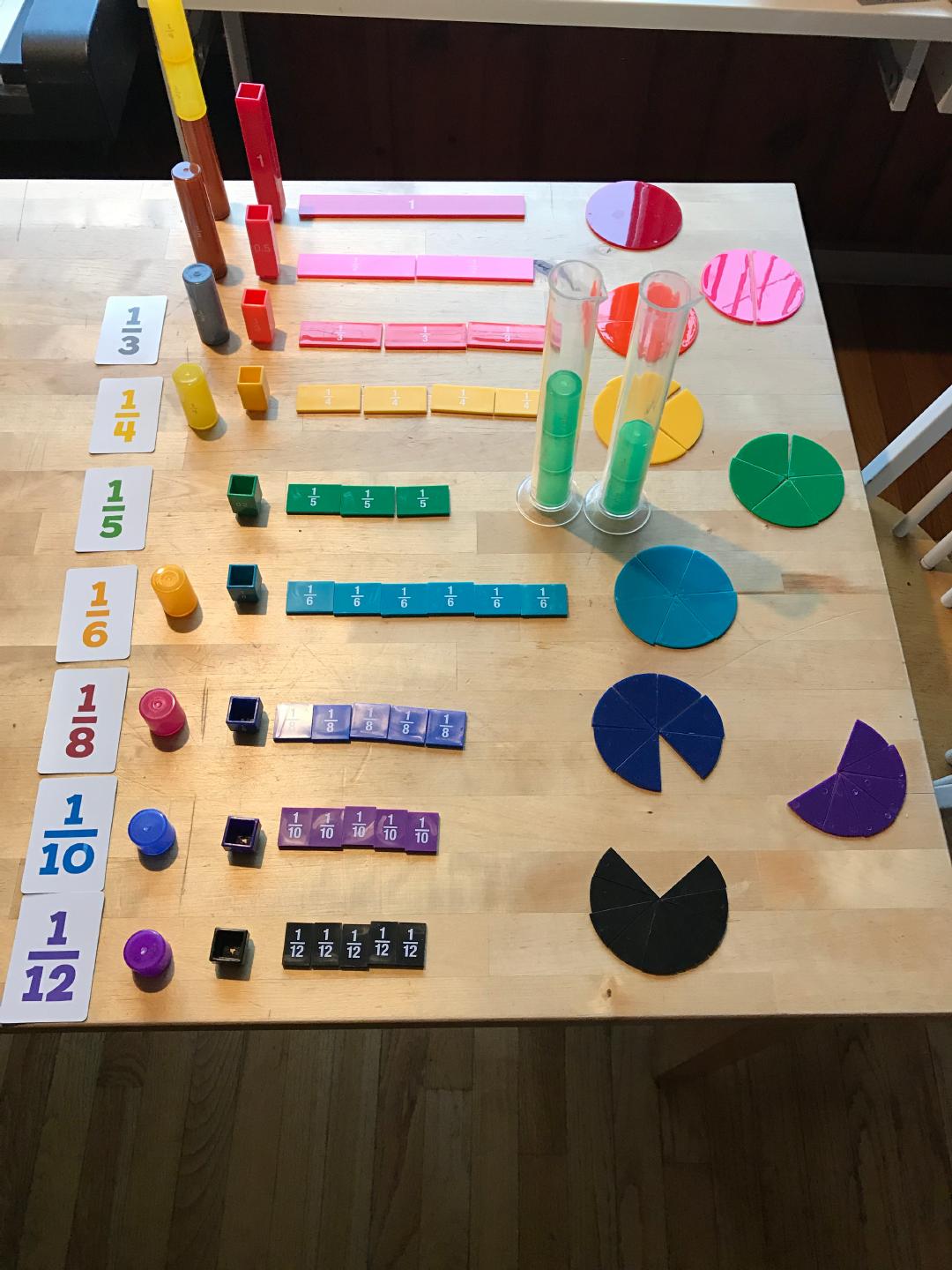}}
\scalebox{0.07}{\includegraphics{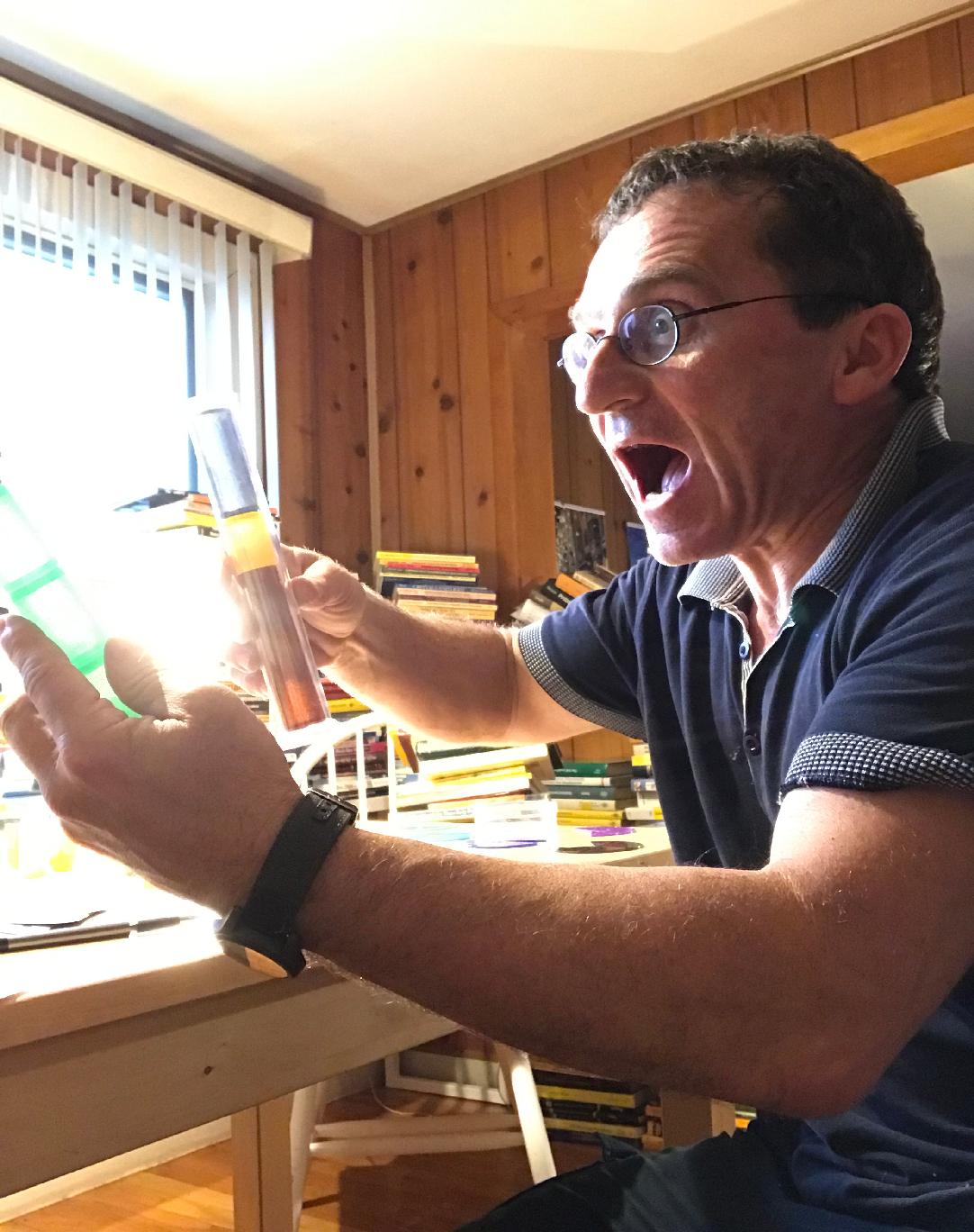}}
\label{Frac lab}
\caption{
An undercover fraction lab, operative in the winter 2016/2017. 
}
\end{figure}

\paragraph{}
During my adventures in algebra, I had not yet known about the Sabidussi ring
and the Shannon ring and was primarily focussed also in the case of ``additive
primes" in the Zykov monoid. Taking graph complements showed there that the
additive primes in the Zykov monoid are the graphs for which the graph complement is 
connected. The question to characterize multiplicative primes in the Shannon ring
is interesting too. I took up the topic in this document again now in the fall of 2024
while teaching a tutorial Math 99R at Harvard on "Visualization of Mathematics" and
thought about how to characterize multiplicative primes in the ring $\mathcal{P}$ of
partitions. (Characterizing additive primes in $\mathcal{P}$ is a nice exercise left 
for the reader.)

\section{Rings of Partitions}

\paragraph{}
To represent partitions uniquely in the ring of partitions, we always sort them. 
Apropos primes, we end this exposition with an open question. How can we characterize 
primes in the ring of partitions? We have seen that $n=(n_1, \dots, n_k)$ is prime if either 
$n=\sum_j n_j$ is prime or if $k$ is prime. But there are more primes like 
$p=(3,4,5,2)$ in which case both $\sum_j n_j=14$ and $k=4$ are not prime, but where $p$
can not be factored into smaller components. The only possible factors would be of the form 
$(a,b), (c,d)$ and since $3,5$ are prime, it would force either $(a,b)=(1,1)$ or $(c,d) = (1,1)$
or $(a,b)=(1,b)$ and $(c,d)=(1,d)$ which all do not work. 

\paragraph{}
This motivates to investigate the growth rate of {\bf partition primes}, similarly as we
are interested in the growth rate of {\bf rational primes}.
Unexplored is whether one can use partitions for cryptology. 
Given a rational integer $N \in \mathbb{N}$, is there a finite commutative ring to 
work with? We have experimented with some set-up. 
One could try to use it for schemes like {\bf Diffie-Hellman key exchange} or 
{\bf RSA public key systems} in such a {\bf modular ring of partition}. 
In any finite Abelian group, there are cryptological algorithms to factor integers could
be used \cite{Riesel}. A simple example is the Pollard $\rho$ algorithm. 
It only requires to build in fast ways to form $a^n$ if $a \in \mathcal{P}$ and 
$n \in \mathbb{Z}$ and to build the greatest common divisor ${\rm gcd}(n,m)$  
of two elements $n,m$ in $\mathcal{P}$.

\bibliographystyle{plain}

\end{document}